\documentclass[a4paper,10pt]{article}
\usepackage[latin1]{inputenc}
\usepackage[T1]{fontenc}
\usepackage{amssymb}
\usepackage{epsfig}
\usepackage{amssymb}
\usepackage{amsmath}
\usepackage{amsfonts}
\usepackage{graphicx, url}
\usepackage{makeidx}
\usepackage{cite}
\usepackage{fancyhdr}
\usepackage{color}
\usepackage{setspace}
\usepackage{bm}

\newtheorem{theorem}{Theorem}[section]

\newtheorem{remark}[theorem]{Remark}

\newenvironment {proof} {{\it Proof.}}{\hspace*{\fill}$\Box$\par\vspace{4mm}}
\usepackage{amsmath}
\usepackage{amsfonts}

\usepackage{setspace}
\usepackage{fancyhdr}
\newcommand{\esssup}{ess\,sup\,}
\newcommand{\mc}{\mathcal}
\newcommand{\mb}{\mathbb}

\title{Forward backward stochastic differential equation games with delay and noisy memory}

\author{K. R. Dahl\footnotemark[2]}

\begin{document}

\maketitle

\renewcommand{\thefootnote}{\fnsymbol{footnote}}

\footnotetext[2]{Department of Mathematics, University of Oslo, Pb. 1053 Blindern, 0316 Oslo, Norway. kristrd@math.uio.no}

\medskip


\begin{abstract}
The goal of this paper is to study a stochastic game connected to a system of forward backward stochastic differential equations (FBSDEs) involving delay and noisy memory. We derive sufficient and necessary maximum principles for a set of controls for the players to be a Nash equilibrium in the game. Furthermore, we study a corresponding FBSDE involving Malliavin derivatives. This kind of equation has not been studied before. The maximum principles give conditions for determining the Nash equilibrium of the game. We use this to derive a closed form Nash equilibrium for an economic model where the players maximize their consumption with respect to recursive utility.

\medskip

\end{abstract}

\section{Introduction}
\label{sec: introduction}

The aim of this paper is to study a stochastic game between two players. The game is based on a forward stochastic differential equation (SDE) for the process $X$. In applications to economy, this process can be thought of as the market situation, e.g. the financial market, the housing market or the oil market. This SDE includes two kinds of memory of the past; regular memory and noisy memory. Regular memory (also called delay, see f. ex. the survey paper by Ivanov et al.~\cite{IvanovEtAl}) means that the SDE can depend on previous values of the process $X$. That is, for some given $\delta > 0$, $X(t)$ depends on $X(t-\delta)$. For more on stochastic delay differential equations and optimal control with delay, see {\O}ksendal et al~\cite{Oksendal_delay_2011} and Agram and {\O}ksendal~\cite{AgramOksendal_delay}. In constrast, noisy memory means that the SDE may involve an It{\^o} integral over previous values of the process, so for $\delta > 0$, $X(t)$ depends on $\int_{t-\delta}^t X(s) dB(s)$ where $\{B(s)\}_{s \in [0,T]}$ is a Brownian motion. For more on noisy memory, see Dahl et al.~\cite{Dahl}.

Connected to this SDE are two backward stochastic differential equations (BSDEs). These BSDEs are connected to the SDE in the sense that they depend on $\{X(t)\}_{t \in [0,T]}$, as well as the delay and noisy memory of this process. Hence, this forms an FBSDE system. Each of these BSDEs corresponds to one of the players in the stochastic game; corresponding to player $i = 1, 2$ is a BSDE in the process $\{W_i(t)\}_{t \in [0,T]}$. The length of memory can be different for the two players, so for $i=1,2$, player $i$ has memory span $\delta_i$. The players may also have different levels of information, which is included in the model by having (potentially) different filtrations $\{\mc{E}^{(i)}_t\}_{t \in [0,T]}$, $i=1,2$.

Each of the players aim to find an optimal control $u_i$ which maximizes their personal performance (objective) function, $J_i$. Seminal work in stochastic optimal control has been done by Krylov and his students, see e.g. Krylov~\cite{Krylov} and \cite{Krylov2}. The performance function of each of the agents will be defined in such a way that it depends on the player's profit rate, the market process $X$ and the process $W_i$ coming from the player's BSDE (more on this in Section~\ref{sec: the_problem}, equation~\eqref{eq: performance}). This kind of problem, where both players maximize their performance which depends on an FBSDE, is called an FBSDE stochastic game, and has been studied by e.g. {\O}ksendal and Sulem~\cite{OS_artikkel}. However, they do not include memory in their model. We study conditions for a pair of controls $(u_1, u_2)$ to be a Nash equilibrium for such a stochastic game. That is, we would like to determine controls such that the players cannot benefit by changing their actions. In order to do so, we derive sufficient and necessary maximum principles giving conditions for a control to be Nash optimal. This is done in Sections~\ref{sec: sufficient} and \ref{sec: necessary}. Maximum principles for forward backward stochastic differential equations (FBSDEs) have been studied by Wang and Wu~\cite{Wu} as well as {\O}ksendal and Sulem~\cite{OS_artikkel}, but these papers do not consider a stochastic game.

In connection with these maximum principles, there are adjoint equations (see e.g. {\O}ksendal~\cite{Oksendal} for an introduction to sotchastic maximum principles and adjoint equations, or {\O}ksendal and Sulem~\cite{OksendalSulemDelay} for maximum principles and adjoint equations where delay is involved). In our case, these adjoint equations are a system of coupled forward backward stochastic differential equations involving Malliavin derivatives (see Di Nunno et al.~\cite{DiNunno} for more on Malliavin derivatives). To the best of our knowledge, such equations have not been studied before. In Section~\ref{sec: FBSDE} we study a slightly simplified version of these adjoint FBSDEs, and establish a connection between these equations and a  system of FBSDEs without Malliavin derivatives. Finally, in Section~\ref{sec: application}, we apply our results to a specific example in order to determine the optimal consumption with respect to recursive utility.



\section{The problem}
\label{sec: the_problem}

Let $(\Omega, \mc{F}, P)$ be a probability space, and let $B(t)$, $t \in [0, T]$ be a Brownian motion in this space. Also, let $\tilde{N}(t, \cdot)$ be an independent compensated Poisson random measure. Let $(\mc{F}_t)_{t \in [0,T]}$ be the $P$-augmented filtration generated by $B(t)$ and $\tilde{N}(t, \cdot)$.

We will consider a game between two players: player $1$ and player $2$. Let $u_i(t)$ be the control process chosen by player $i=1,2$, and denote $\bm{u}(t) = (u_1(t), u_2(t))$. Let $\mc{A}_i$, $i=1,2$, denote the set of admissible controls for player $i$ and $\mc{A} = \mc{A}_1 \times \mc{A}_2$.

We consider a controlled forward stochastic differential equation for a process $X(t)=X_u(t,\omega)$, $\omega \in \Omega, t \in [0,T]$ determining the market situation (in the following, we omit the $\omega$ for notational ease unless it is important to highlight its dependence):

\begin{equation}
\label{eq: FSDE}
\begin{array}{lll}
dX(t) &=& b(t, X(t), \bm{Y}(t), \bm{\Lambda}(t), \bm{u}(t), \omega) dt \\[\smallskipamount]
&& + \sigma(t, X(t), \bm{Y}(t), \bm{\Lambda}(t), \bm{u}(t), \omega) dB(t) \\[\smallskipamount]
&& + \int_{\mb{R}} \gamma(t^-, X(t^-), \bm{Y}(t^-), \bm{\Lambda}(t^-), \bm{u}(t^-), \zeta, \omega) \tilde{N}(dt, d \zeta) \\[\medskipamount]
X(0) &=& x
\end{array}
\end{equation}
\noindent where $\bm{Y}(t) = (Y_1(t), Y_2(t))$, $\bm{\Lambda}(t) = (\Lambda_1(t), \Lambda_2(t))$, and $Y_i(t) := X(t - \delta_i)$, $\Lambda_i(t) := \int_{t - \delta_i}^t X(s) dB(s)$, and $\delta_i \geq 0$ for $i=1,2$. The superscript $t^-$ means that we are taking the left limit of the process is question (that is, the value before a potential jump at time $t$), see {\O}ksendal and Sulem~\cite{OksendalSulemJump} for more on this.

Here, the \emph{delay processes} $Y_i$, and the \emph{noisy memory processes} $\Lambda_i$ correspond to player $i =1,2$ respectively. Hence, the two players may have memories for different time intervals, depending on the values of $\delta_i$.  Also,

\[
 \begin{array}{lll}
b : [0,T] \times \mb{R} \times \mb{R}^2 \times \mb{R}^2 \times \mc{A} \times \Omega \rightarrow \mb{R}, \\[\smallskipamount]
\sigma: [0,T] \times \mb{R} \times \mb{R}^2 \times \mb{R}^2 \times \mc{A} \times \Omega \rightarrow \mb{R}, \\[\smallskipamount]
\gamma : [0,T] \times \mb{R} \times \mb{R}^2 \times \mb{R}^2 \times \mc{A} \times \Omega \rightarrow \mb{R} \\[\smallskipamount]
 \end{array}
\]
\noindent are predictable functions such that for each $\bm{u} \in \mc{A}$ the SDE \eqref{eq: FSDE} has a unique solution.

\begin{remark}
Existence and uniqueness of solution for the SDE \eqref{eq: FSDE} is guaranteed under certain, fairly unrestrictive, assumptions on the coefficient functions, see Dahl et al.~\cite{Dahl}, Assumption 1, for conditions ensuring existence and uniqueness of solution to \eqref{eq: FSDE}. This can be seen by viewing equation \eqref{eq: FSDE} as a stochastic functional differential equation.
\end{remark}

In addition to this, the players (potentially) have different levels of information, represented by different subfiltrations $\mc{E}^{(i)}_t \subseteq \mc{F}_t$ for all $t \in [0, T]$, $i=1,2$.

For $i=1,2$, let $g_i(\cdot, x, y, \Lambda, w_i, z_i, k_i(\cdot), u, \omega)$ be a given predictable process, and let $h_i(x, \omega)$ be an $\mc{F}_T$-measurable function. Associated to the FSDE~\eqref{eq: FSDE}, we have a pair of backward stochastic differential equations (BSDEs) in the unknown stochastic processes $(W_i, Z_i, K_i)$, $i=1,2$:

\begin{equation}
\label{eq: BSDE}
\begin{array}{lll}
dW_i(t) &=& -g_i(t, X(t), \bm{Y}(t), \bm{\Lambda}(t), W_i(t), Z_i(t), K_i(t, \cdot), \bm{u}(t), \omega) dt \\[\smallskipamount]
&& + Z_i(t) dB(t) + \int_{\mb{R}} K_i(t, \zeta) \tilde{N}(dt, d\zeta) \\[\medskipamount]
W_i(T) &=& h_i(X(T), \omega).
\end{array}
\end{equation}

Note that these BSDEs are coupled to the SDE~\eqref{eq: FSDE} due to the dependency on $X$. Also, the BSDEs depend on the memory of the market process $X$, due to the dependency on the processes $\bm{Y}$ and $\bm{\Lambda}$. However, equation~\eqref{eq: BSDE} is a standard BSDE, hence the conditions for existence and uniqueness of solution is well known, see e.g. Pardoux and Peng \cite{PardouxPeng}.


For $i=1,2$, let $f_i: [0,T] \times \mb{R} \times \mb{R} \times \mb{R} \times \mc{A} \times \Omega \rightarrow \mb{R}, \varphi_i : \mb{R} \rightarrow \mb{R}, \psi_i : \mb{R} \rightarrow \mb{R}$ be functions representing a profit rate, bequest function and risk evaluation. Then, the performance function of each player $i=1,2$ is defined by:

\begin{equation}
\label{eq: performance}
J_i(u) = E[\int_0^T f_i(t, X^u(t), Y_i^u(t), \Lambda_i^u(t), u_i(t)) dt + \varphi_i(X^u(T)) + \psi_i(W_i^u(0))]
\end{equation}

\noindent where we must assume all conditions necessary for the integrals and the expectation to exist.

Also, note that the performance $J_i$ of player $i$ is a function of the control $\bm{u}(t) = (u_1(t), u_2(t))$, which is determined by both players. Therefore, this problem setting specifies a stochastic game.

A pair of controls $(\hat{u}_1, \hat{u}_2)$ is called a \emph{Nash equilibrium} for this stochastic game if the following holds:

\begin{equation}
\label{eq: Nash}
\begin{array}{lll}
J_1(u_1, \hat{u}_2) &\leq& J_1(\hat{u}_1, \hat{u}_2) \mbox{ for all } u_1 \in \mc{A}_1 \\[\smallskipamount]
J_2(\hat{u}_1, u_2) &\leq& J_2(\hat{u}_1, \hat{u}_2) \mbox{ for all } u_2 \in \mc{A}_2.
\end{array}
\end{equation}

In words, this means that in the Nash equilibrium, neither player would like to change their control.

Assume there exists a Nash equilibrium for this forward-backward stochastic differential (FBSDE) game with delay and noisy memory. We would like to find this Nash equilibrium, and we will do so by proving sufficient and necessary maximum principles for this problem. Therefore, we define a Hamiltonian function for each player $i =1,2$ as follows:

\begin{equation}
\label{eq: Hamiltonians}
\begin{array}{lll}
H_i(t,x, \bm{y}, \bm{\Lambda}, w_i, z_i, k_i, u_1, u_2, \lambda_i, p_i, q_i, r_i) = f_i(t,x, y_i, \Lambda_i, u_i) \\[\smallskipamount]
\quad \quad + \lambda_i g_i(t,x, \bm{y}, \bm{\Lambda}, w_i, z_i, k_i, u_1, u_2) + p_i b(t,x, \bm{y}, \bm{\Lambda}, u_1, u_2) \\[\smallskipamount]
\quad \quad + q_i \sigma(t,x, \bm{y}, \bm{\Lambda}, u_1, u_2)+ \int_{\mb{R}} r_i(\zeta) \gamma (t,x, \bm{y}, \bm{\Lambda}, u_1, u_2, \zeta) \nu(d \zeta).
\end{array}
\end{equation}

Assume $H_i$ is $C^1$ in $x,y_1,y_2, \Lambda_1, \Lambda_2, w_i, z_i, k_i, u_1, u_2$ for $i=1,2$. In the following, for ease of notation, we will use the abbreviation
\[
H_i(t)=H_i(t,x, \bm{y}, \bm{\Lambda}, w_i, z_i, k_i, u_1, u_2, \lambda_i, p_i, q_i, r_i)
\]

For $i=1,2$, we define a system of FBSDEs associated to these Hamiltonians in the unknown adjoint processes $(\lambda_i, p_i, q_i, r_i)$:

FSDE in $\lambda_i$ (which depends on $p_i, q_i, r_i$):
\begin{equation}
\label{eq: FSDE_adjoint}
\begin{array}{lll}
d \lambda_i(t) &=& \frac{\partial H_i}{\partial w_i}(t) dt + \frac{\partial H_i}{\partial z_i}(t) d B(t) + \int_{\mb{R}} \nabla_{k_i}(H_i(t, \zeta)) \tilde{N}(dt, d\zeta) \\[\smallskipamount]
\lambda_i(0) &=& \psi_{i}'(W_i(0)).
\end{array}
\end{equation}

\noindent where $\nabla_{k_i}(H_i(t, \zeta))$ is the Fr\'{e}chet derivative of $H_i$ at $k_i$, see the appendix in {\O}ksendal and Sulem~\cite{OS_artikkel} for a closer explanation of this gradient. 

We also define a BSDE in $p_i, q_i, r_i$, which depends on $\lambda_i$:

\begin{equation}
\label{eq: BSDE_adjoint}
\begin{array}{lll}
dp_i(t) &=& E[\mu_i(t) | \mc{F}_t]dt + q_i(t) dB(t) + \int_{\mb{R}} r_i(t, \zeta) \tilde{N}(dt, d\zeta) \\[\smallskipamount]
p_i(T) &=& \varphi_{i}'(X(T)) + h_{i}^{'}(X(T)) \lambda_i(T)
\end{array}
\end{equation}

\noindent where

\[
\mu_i(t) = -\frac{\partial H_i}{\partial x}(t) - \frac{\partial H_i}{\partial y_i}(t+\delta_i) \boldsymbol{1}_{[0, T-\delta_i]}(t) - \int_t^{t+\delta_i} D_t[\frac{\partial H_i}{\partial \Lambda_i}(s) \boldsymbol{1}_{[0,T]}(s) ds]
\]

\noindent and $D_t[\cdot]$ denotes the Malliavin derivative (see Remark~\ref{remark: Malliavin}). Note that the conditional expectation in \eqref{eq: BSDE_adjoint} is well defined by the extension of the Malliavin derivative introduced by Aase et al.~\cite{AaseEtAl}, see Remark~\ref{remark: Malliavin}. Equations~\eqref{eq: FSDE_adjoint}-\eqref{eq: BSDE_adjoint} form an FBSDE-system involving Malliavin derivatives. To the best of our knowledge, such systems have not been studied before.

\begin{remark}
\label{remark: Malliavin}
We refer to Nualart~\cite{Nualart}, Sanz-Sol\`{e}~\cite{SanzSole} and Di Nunno et al.~\cite{DiNunno} for information about the Malliavin derivative $D_t$ for Brownian motion $B(t)$ and, more generally, L\'{e}vy processes. In Aase et al.~\cite{AaseEtAl}, $D_t$ was extended from the space $\mb{D}_{1,2}$ to $L^2(P)$, where $\mb{D}_{1,2}$ denotes the classical space of Malliavin differentiable $\mc{F}_T$-measurable random variables. The extension is such that for all $F \in L^2(\mc{F}_T, P)$, the following holds:

\begin{enumerate}
\item[$(i)$] $D_t F \in (\mc{S})^*$, where $(\mc{S})^* \supseteq L^2(P)$ denotes the Hida space of stochastic distributions,

\item[$(ii)$] the map $(t,\omega) \rightarrow E[D_t F | \mc{F}_t]$ belongs to $L^2(\mc{F}_T, \lambda \times P)$, where $\lambda$ denotes the Lebesgue measure on $[0,T]$.

Moreover, the following \emph{generalized Clark-Ocone theorem} holds:

\item[$(iii)$]
\begin{equation}
F = E[F] + \int_0^T E[D_t F | \mc{F}_t] dB(t).\label{Clark-Ocone}
\end{equation}
See \cite{AaseEtAl}, Theorem 3.11, and also \cite{DiNunno}, Theorem 6.35.

Notice that  by combining It\^o's isometry with the Clark-Ocone theorem, we obtain
\begin{equation}\label{eq:normOfDtF-varF}
 E\Big[\int_0^T E[D_t F|\mathcal F_t]^2 dt\Big]= E\Big[\Big(\int_0^T E[D_t F|\mathcal F_t]dB(t)\Big)^2\Big]=E[(F^2-E[F]^2)]
\end{equation}

\item[$(iv)$]
As observed in Agram et al.~\cite{AgramOksendal}, we can also apply the Clark-Ocone theorem to show the following generalized duality formula:

Let $F \in L^2(\mc{F}_T, P)$ and let $\varphi(t) \in L^2(\lambda \times P)$ be adapted. Then

\begin{equation}
\label{eq: gen-duality}
E \Big[ F \int_0^T \varphi(t) dB(t) \Big] = E \Big[ \int_0^T E[D_t F | \mc{F}_t] \varphi(t) dt \Big]
\end{equation}

\end{enumerate}

\end{remark}

\begin{remark}

Note that equation~\eqref{eq: FSDE_adjoint} is linear in $\lambda_i$, and hence, if $p_i, q_i, r_i$ were given, it could be solved by using the It{\^o} formula. However, this solution will depend on the processes $X, Y_i, \Lambda_i$ and $W_i$, so in order to find an explicit solution for $\lambda_i$, we must also solve the coupled FBSDE system \eqref{eq: FSDE}-\eqref{eq: BSDE}. 

The BSDE~\ref{eq: BSDE_adjoint} is linear in $p_i$, and hence, if $\lambda_i$ was given, it would be possible to find a unique solution to this equation by using e.g. Proposition 6.2.1 in Pham~\cite{Pham} or Theorem 1.7 in {\O}ksendal and Sulem~\cite{OksendalSulem_riskmin}. However, as for the adjoint SDE~\eqref{eq: FSDE_adjoint}, this solution will depend on the coupled FBSDE system \eqref{eq: FSDE}-\eqref{eq: BSDE}.

\end{remark}


In the remaining part of the paper, we will prove a sufficient (Section~\ref{sec: sufficient}) and a necessary maximum principle (Section~\ref{sec: necessary}) for this kind of FBSDE game with delay and noisy memory. Then, we will study existence and uniqueness of solutions of the FBSDE system \eqref{eq: FSDE_adjoint}-\eqref{eq: BSDE_adjoint} (Section~\ref{sec: FBSDE}). Finally, we will present an example which illustrates our results: optimal consumption rate with respect to recursive utility (see Section~\ref{sec: application}).

\section{Sufficient maximum principle for FBSDE games with delay and noisy memory}
\label{sec: sufficient}



We prove a sufficient maximum principle which roughly states that under concavity conditions, a control $(\hat{u}_1, \hat{u}_2)$ satisfying a conditional maximum principle and an $\mc{L}^2$-condition is a Nash equilibrium for the stochastic game.

\begin{theorem}
 \label{thm: Suff-max-princ-FBSDE}

 Let $\hat{u}_1 \in \mc{A}_1$ and  $\hat{u}_2 \in \mc{A}_2$ with corresponding solutions $\hat{X}(t), \hat{Y}_i(t), \hat{\Lambda}_i(t)$, $\hat{W}_i(t), \hat{Z}_i(t),$ $\hat{K}_i(t), \hat{\lambda}_i(t),$ $\hat{p}_i(t), \hat{q}_i(t), \hat{r}_i(t,\zeta)$ of the FSDE~\eqref{eq: FSDE}, the BSDE~\eqref{eq: BSDE}, and the FBSDE system~\eqref{eq: FSDE_adjoint}-\eqref{eq: BSDE_adjoint} for $i=1,2$. Also, assume that:

 \begin{itemize}
  \item{(Concavity I) The functions $x \rightarrow h_i(x)$, $x \rightarrow \varphi_i(x), x \rightarrow \psi_i(x)$ are concave for $i=1,2$.}
  \item{(The conditional maximum principle)
  \[
   \begin{array}{lll}
    \esssup_{v \in \mc{A}_1} E[H_1(t,\hat{X}(t), \hat{\bm{Y}}(t), \hat{\bm{\Lambda}}(t), \hat{W}_1(t), \hat{Z}_1(t), \hat{K}_1(t,\cdot), \\[\smallskipamount]
    \hspace{1.5cm}v, \hat{u}_2(t), \hat{\lambda_1}(t), \hat{p}_1(t), \hat{q}_1(t), \hat{r}_1(t,\cdot)) | \mc{E}_t^{(1)}] \\[\smallskipamount]
    = E[H_1(t,\hat{X}(t), \hat{\bm{Y}}(t), \hat{\bm{\Lambda}}(t), \hat{W}_1(t), \hat{Z}_1(t), \hat{K}_1(t,\cdot), \\[\smallskipamount]
    \hspace{1.5cm} \hat{u}_1(t), \hat{u}_2(t), \hat{\lambda}_1(t), \hat{p}_1(t), \hat{q}_1(t), \hat{r}_1(t,\cdot)) | \mc{E}_t^{(1)}]
   \end{array}
  \]
  \noindent and similarly
  \[
   \begin{array}{lll}
    \esssup_{v \in \mc{A}_2} E[H_2(t,\hat{X}(t),\hat{\bm{Y}}(t), \hat{\bm{\Lambda}}(t), \hat{W}_2(t), \hat{Z}_2(t), \hat{K}_2(t,\cdot), \\[\smallskipamount]
    \hspace{1.5cm} \hat{u}_1, v, \hat{\lambda_2}(t), \hat{p}_2(t), \hat{q}_2(t), \hat{r}_2(t,\cdot)) | \mc{E}_t^{(2)}] \\[\smallskipamount]
     = E[H_2(t,\hat{X}(t),\hat{\bm{Y}}(t), \hat{\bm{\Lambda}}(t), \hat{W}_2(t), \hat{Z}_2(t), \hat{K}_2(t,\cdot), \\[\smallskipamount]
    \hspace{1.5cm} \hat{u}_1(t), \hat{u}_2(t), \hat{\lambda}_2(t), \hat{p}_2(t), \hat{q}_2(t), \hat{r}_2(t,\cdot)) | \mc{E}_t^{(2)}].
   \end{array}
  \]
}

\item{(Concavity II) The functions
\[
 \begin{array}{lll}
  \hat{\mc{H}}_1(t, x, y_1, \Lambda_1, w_1, z_1, k_1) \\[\smallskipamount]
:= \esssup_{v \in \mc{A}_1} E[H_1(t,x, y_1, \hat{y}_2,\Lambda_1, \hat{\Lambda}_2, w_1, z_1, k_1, v, \hat{u}_2, \hat{\lambda}_1, \hat{p}_1, \hat{q}_1, \hat{r}_1) | \mc{E}_t^{(1)}]
 \end{array}
\]
\noindent and
\[
 \begin{array}{lll}
  \hat{\mc{H}}_2(t, x, y_2, \Lambda_2, w_2, z_2, k_2) \\[\smallskipamount]
:= \esssup_{v \in \mc{A}_2} E[H_2(t,x, \hat{y}_1, y_2,\hat{\Lambda}_1, \Lambda_2, w_2, z_2, k_2, \hat{u}_1, v, \hat{\lambda}_2, \hat{p}_2, \hat{q}_2, \hat{r}_2) | \mc{E}_t^{(2)}]
 \end{array}
\]
\noindent are concave for all $t$ a.s.
}
\item{Finally, assume that the following $\mc{L}^2$ conditions hold:

\[
 \begin{array}{lll}
  E[\int_0^T \Big\{ \hat{p}_i^2(t) \Big[ \big(\sigma(t) - \hat{\sigma}(t)\big)^2 + \int_{\mb{R}} \big( r_i(t,\zeta) - \hat{r}_i(t,\zeta) \big)^2 \nu(d\zeta)    \Big] \\[\smallskipamount]
 \hspace{0.5cm} + \big(X(t) - \hat{X}(t)\big)^2 [\hat{q}_i^2(t) + \int_{\mb{R}} \hat{r}_i^2(t,\zeta) \nu(d\zeta)] \\[\smallskipamount]
 \hspace{0.5cm} + \big(Y_i(t) - \hat{Y}_i(t)\big)^2 [(\frac{\partial \hat{H}_i}{\partial z})^2(t) + \int_{\mb{R}} ||\nabla_k \hat{H}_i(t,\zeta) ||^2 \nu(d\zeta)] \\[\smallskipamount]
\hspace{0.5cm} + \hat{\lambda}_i^2(t) [ \big(\Lambda_i(t) - \hat{\Lambda}_i(t)\big)^2 + \int_{\mb{R}} (K_i(t,\zeta) - \hat{K}_i(t,\zeta))^2 \nu(d\zeta)] \Big\}] < \infty
 \end{array}
\]
\noindent for $i=1,2$.
}
\end{itemize}

 Then, $(\hat{u}_1, \hat{u}_2)$ is a Nash equilibrium.
\end{theorem}

\begin{proof}
 We would like to show that $J_1(u_1, \hat{u}_2) \leq J_1(\hat{u}_1, \hat{u}_2)$ for all $u_1 \in \mc{A}_1$. Choose $u_1 \in \mc{A}_1$. By the definition of the performance function $J_1$,

 \[
  \delta := J_1(u_1, \hat{u}_2) - J_1(\hat{u}_1, \hat{u}_2) = I_1 + I_2 + I_3
 \]
\noindent where

\[
 I_1 = E[\int_0^T \{f_1(t, x, y, \Lambda, \bm{u}) - f_1(t, \hat{x}, \hat{y}, \hat{\Lambda}, \hat{\bm{u}})\}dt],
\]

\[
 I_2= E[\varphi_1(X(T)) - \varphi_1(\hat{X}(T))],
\]

\[
 I_3= E[\psi_1(W_1(0)) - \psi_1(\hat{W}_1(0))].
\]

Note that from the definition of the Hamiltonian,

\begin{equation}
\label{eq: I_1}
\begin{array}{llll}
 I_1 &=& E[\int_0^T \{ H_1(t) - \hat{H}_1(t) - \hat{\lambda}_1(t) (g_1(t) - \hat{g}_1(t)) - \hat{p}_1(t)(b(t) - \hat{b}(t)) \\[\smallskipamount]
 &&- \hat{q}_1(t)(\sigma(t) - \hat{\sigma}(t)) - \int_{\mb{R}} \hat{r}_1(t, \zeta) (\gamma(t, \zeta) - \hat{\gamma}(t, \zeta) \nu(d \zeta)) \}dt]
 \end{array}
\end{equation}
\noindent where we have used the abbreviation
\[
 \hat{H}_1(t) := H_1(t, \hat{X}(t), \hat{\bm{Y}}(t), \hat{\bm{\Lambda}}(t), \hat{W}_1(t), \hat{Z}_1(t), \hat{K}_1(t, \cdot), \hat{\bm{u}}, \hat{\lambda}_1, \hat{p}_1, \hat{q}_1, \hat{r}_1, \omega)
\]
\noindent and corresponding abbreviations for $H_1(t), b(t), \hat{b}(t), \sigma, \hat{\sigma}(t), \gamma(t)$ and $\hat{\gamma}(t)$.

Also,

\begin{equation}
 \label{eq: I_2}
 \begin{array}{llll}
  I_2 &=& E[\varphi_1(X(T)) - \varphi_1(\hat{X}(T))] \\[\smallskipamount]

  &\leq& E[\varphi_1'(\hat{X}(T))(X(T)- \hat{X}(T))] \\[\smallskipamount]

  &=& E[(\hat{p}_1(T) - h_1'(\hat{X}(T))\hat{\lambda}_1(T)) (X(T) - \hat{X}(T))] \\[\smallskipamount]

  &=& E[\hat{p}_1(T)(X(T) - \hat{X}(T))] - E[\hat{\lambda}_1(T)h_1'(\hat{X}(T))(X(T) - \hat{X}(T))] \\[\smallskipamount]

  &=& E[\int_0^T  \hat{p}_1(t) (dX(t) - d\hat{X}(t)) + \int_0^T (X(t) - \hat{X}(t))d\hat{p}_1(t) \\[\smallskipamount]

  &&+ \int_0^T \hat{q}_1(t)(\sigma(t) - \hat{\sigma}(t)) dt + \int_0^T \int_{\mb{R}} \hat{r}_1(t, \zeta) (\gamma(t, \zeta) - \hat{\gamma}(t, \zeta))  \nu(d \zeta) dt] \\[\smallskipamount]

  &&- E[\hat{\lambda}_1(T) h_1'(\hat{X}(T))(X(T) - \hat{X}(T))] \\[\smallskipamount]

  &=& E[\int_0^T \hat{p}_1(t)(b(t) - \hat{b}(t))dt + \int_0^T (X(t) - \hat{X}(t))(- \frac{\partial \hat{H}_1}{\partial x}(t) \\[\smallskipamount]

  &&- \frac{\partial \hat{H}_1}{\partial y_1}(t + \delta_1) \boldsymbol{1}_{[0,T- \delta_1]}(t) + \int_t^{t+\delta_1} D_t[- \frac{\partial \hat{H}_1}{\partial \Lambda_1}(s)] \boldsymbol{1}_{[0,T]}(s)ds  )dt \\[\smallskipamount]

  &&+ \int_0^T \hat{q}_1(t)(\sigma(t) - \hat{\sigma}(t)) dt + \int_0^T \int_{\mb{R}} \hat{r}_1(t, \zeta) (\gamma(t, \zeta) - \hat{\gamma}(t, \zeta))  \nu(d \zeta) dt]]\\[\smallskipamount]

  &&- E[\hat{\lambda}_1(T) h_1'(\hat{X}(T))(X(T) - \hat{X}(T))] \\[\smallskipamount]
 \end{array}
\end{equation}
\noindent where the first inequality follows from the concavity of $\varphi_1$, the second equality follows from equation~\eqref{eq: BSDE_adjoint}, the fourth equality from It{\^o}'s product rule applied to $\hat{p}_1 X$ and $\hat{p}_1 \hat{X}$, the fifth equality follows from equation~\eqref{eq: BSDE_adjoint}, the double expectation rule and equation~\eqref{eq: FSDE}.

Also, note that

\begin{equation}
 \label{eq: I_3}
 \begin{array}{llll}
  I_3 &=& E[\psi_1(W_1(0)) - \psi_1(\hat{W}_1(0))] \\[\smallskipamount]

  &\leq& E[\psi_1'(\hat{W}_1(0))(W_1(0) - \hat{W}_1(0))] \\[\smallskipamount]

  &=& E[\hat{\lambda}_1(T) (W_1(T) - \hat{W}_1(T))] - \{E[\int_0^T (W_1(t)- \hat{W}_1(t)) d \hat{\lambda}_1(t) \\[\smallskipamount]

  &&+ \int_0^T \hat{\lambda}_1(t) (dW_1(t) - d \hat{W}_1(t)) + \int_0^T \frac{\partial \hat{H}_1}{\partial z_1}(t) (Z_1(t) - \hat{Z}_1(t)) dt \\[\smallskipamount]

  &&+ \int_0^T \int_{\mb{R}} \nabla_{k_1} \hat{H}_1(t) (K_1(t) - \hat{K}_1(t)) \nu(d \zeta) dt]\} \\[\smallskipamount]

  &=& E[\hat{\lambda}_1(T)(h_1(X(T)) - h_1(\hat{X}(T)))] - \{ E[\int_0^T \frac{\partial \hat{H}_1}{\partial w_1}(t) (W_1(t) - \hat{W}_1(t)) dt \\[\smallskipamount]

  &&+ \int_0^T \hat{\lambda}_1(t) (-g_1(t) + \hat{g}_1(t)) dt + \int_0^T \frac{\partial \hat{H}_1}{\partial z_1}(t) (Z_1(t) - \hat{Z}_1(t)) dt \\[\smallskipamount]

  &&+ \int_0^T \int_{\mb{R}} \nabla_k \hat{H}_1(t) (K_1(t) - \hat{K}_1(t)) \nu(d \zeta) dt ] \} \\[\smallskipamount]

  &\leq& E[\hat{\lambda}_1(T) h_1'(\hat{X}(T)) (X(T) - \hat{X}(T))] -  \{ E[\int_0^T \frac{\partial \hat{H}_1}{\partial w_1}(t) (W_1(t) - \hat{W}_1(t)) dt \\[\smallskipamount]

  &&+ \int_0^T \hat{\lambda}_1(t) (-g_1(t) + \hat{g}_1(t)) dt + \int_0^T \frac{\partial \hat{H}_1}{\partial z_1}(t) (Z_1(t) - \hat{Z}_1(t)) dt \\[\smallskipamount]

  &&+ \int_0^T \int_{\mb{R}} \nabla_{k_1} \hat{H}_1(t) (K_1(t) - \hat{K}_1(t)) \nu(d \zeta) dt ] \} \\[\smallskipamount]

 \end{array}
\end{equation}

\noindent where the first inequality follows from the concavity of $\psi_1$, the second equality follows from equation~\eqref{eq: FSDE_adjoint}, the third equality follows from It{\^o}'s product rule applied to $\hat{\lambda}_1 Y_1$ and $\hat{\lambda}_1 \hat{Y}_1$, the fourth equality follows from  equation~\eqref{eq: BSDE} as well as equation~\eqref{eq: FSDE_adjoint}. The final inequality follows from the concavity of $h_1$ and that $\hat{\lambda}_1(T) \geq 0$.

Hence,

\begin{equation}
\label{eq: delta}
\begin{array}{lll}
 \Delta &=& I_1 + I_2 + I_3 \\[\smallskipamount]

 &\leq& E[\int_0^T \{H_1(t)- \hat{H}_1(t) - \Big( \frac{\partial \hat{H}_1}{\partial x}(t) + \frac{\partial \hat{H}_1}{\partial y_1}(t + \delta_1) \boldsymbol{1}_{[0,T - \delta_1]}(t) \\[\smallskipamount]

 &&+ \int_t^{t+\delta_1} D_t[\frac{\partial \hat{H}_1}{\partial \Lambda_1}(s)] \boldsymbol{1}_{[0,T]}(s) ds \Big) (X(t) - \hat{X}(t)) dt \} \\[\smallskipamount]

 &&- \int_0^T \big\{ \frac{\partial \hat{H}_1}{\partial w_1}(t)(W_1(t) - \hat{W}_1(t))  + \frac{\partial \hat{H}_1}{\partial z_1}(t)(Z_1(t) - \hat{Z}_1(t)) \\[\smallskipamount]

 &&+ \int_{\mb{R}} \nabla_{k_1} \hat{H}_1(t)(K_1(t, \zeta) - \hat{K}_1(t, \zeta)) \nu(d \zeta) \big\} dt ].
\end{array}
\end{equation}

Note that by changing the order of integration and using the duality formula for Malliavin derivatives (see Di Nunno et al.~\cite{DiNunno}), we get:

\begin{equation}
\label{eq: utregning_lambda}
\begin{array}{lll}
 E \Big[ \int_0^T \frac{\partial{\hat{H}_1}}{\partial{\Lambda_1}}(s) \big(\Lambda_1(s) - \hat{\Lambda}_1(s)\big)ds\Big] \\[\smallskipamount]

\hspace{2cm} = E\Big[\int_0^T \frac{\partial{\hat{H}_1}}{\partial{\Lambda_1}}(s)  \int_{s - \delta_1}^s \big(X(t) - \hat{X}(t)\big)dB(t) ds\Big] \\[\smallskipamount]

\hspace{2cm} = \int_0^T E\Big[\frac{\partial{\hat{H}_1}}{\partial{\Lambda_1}}(s)  \int_{s - \delta_1}^s \big(X(t) - \hat{X}(t)\big)dB(t)\Big] ds  \\[\smallskipamount]

\hspace{2cm} = \int_0^T E[\int_{s - \delta_1}^s E[ D_t(\frac{\partial{\hat{H}_1}}{\partial{\Lambda_1}}(s)) | \mc{F}_t] \big(X(t) - \hat{X}(t)\big) dt] ds \\[\smallskipamount]

\hspace{2cm} =  E[\int_0^T \int^{t + \delta_1}_t E[D_t(\frac{\partial{\hat{H}_1}}{\partial{\Lambda_1}}(s)) | \mc{F}_t] \boldsymbol{1}_{[0,T]}(s) ds (X(t) - \hat{X}(t)) dt\Big]\\[\smallskipamount]

\hspace{2cm} =  E[\int_0^T \int^{t + \delta_1}_t D_t(\frac{\partial{\hat{H}_1}}{\partial{\Lambda_1}}(s)) \boldsymbol{1}_{[0,T]}(s) ds (X(t) - \hat{X}(t)) dt\Big].
\end{array}
\end{equation}

Also, note that

\begin{equation}
 \label{eq: utregning_y}
\begin{array}{lllll}
 E \Big[\int_0^T \frac{\partial{\hat{H}}}{\partial{y_1}}(t) \big(Y_1(t) - \hat{Y_1}(t)\big)dt\Big] \\[\smallskipamount]

\hspace{2cm} = E \Big[\int_0^T \frac{\partial{\hat{H}}}{\partial{y_1}}(t) \big(X(t - \delta) - \hat{X}(t - \delta_1)\big)dt\Big] \\[\smallskipamount]

\hspace{2cm} = E \Big[\int_0^T \frac{\partial{\hat{H}}}{\partial{y_1}}(t + \delta_1) \boldsymbol{1}_{[0,T-\delta_1]}(t) \big(X(t) - \hat{X}(t)\big) dt \Big].
\end{array}
\end{equation}

Hence, by the inequality~\eqref{eq: delta} combined with equations~\eqref{eq: utregning_lambda} and \eqref{eq: utregning_y},

\begin{equation}
\label{eq: delta_2}
\begin{array}{lll}
  \Delta &\leq  E[\int_0^T \{H_1(t)- \hat{H}_1(t) - \frac{\partial \hat{H}_1}{\partial x}(t)(X(t) - \hat{X}(t)) - \frac{\partial \hat{H}_1}{\partial y_1}(t)(Y_1(t) - \hat{Y}_1(t)) \\[\smallskipamount]

 &- \frac{\partial \hat{H}_1}{\partial \Lambda_1}(t) (\Lambda_1(t) - \hat{\Lambda}_1(t)) dt - \frac{\partial \hat{H}_1}{\partial w_1}(t)(W_1(t) - \hat{W}_1(t))  - \frac{\partial \hat{H}_1}{\partial z_1}(t)(Z_1(t) - \hat{Z}_1(t)) \\[\smallskipamount]

 &+ \int_{\mb{R}} \nabla_{k_1} \hat{H}_1(t)(K_1(t, \zeta) - \hat{K}_1(t, \zeta)) \nu(d \zeta) \} dt ].
 \end{array}
\end{equation}

By assumption, $\hat{\mc{H}}_1$ is concave, so it is superdifferentiable\footnote{Defined similarly as subdifferentiability for convex functions.} (see Rockafellar~\cite{Rockafellar}) at the point $\vec{x}$ $:= (\hat{X}, \hat{Y}_1, \hat{\Lambda}_1, \hat{W}_1, \hat{Z}_1, \hat{K}_1)$. Thus, there exists a supergradient $\vec{a} := (a_0, a_1, a_2, a_3, a_4, a_5(\cdot))$ such that for all $\vec{y} := (x,y,\Lambda, w, z, k)$, the following holds:

\begin{equation}
\label{eq: superdiff}
\hat{\mc{H}}_1(\vec{x}) + \vec{a} \cdot (\vec{y} - \vec{x}) \geq \hat{\mc{H}}_1(\vec{y}).
\end{equation}

Define

\begin{equation}
 \label{eq: phi_1}
\begin{array}{lll}
 \phi_1(x, y, \Lambda,  w, z, k) := \hat{\mc{H}}_1(x,y, \Lambda, w, z, k) - \hat{\mc{H}}_1(\hat{X},\hat{Y}_1, \hat{\Lambda}_1, \hat{W}_1, \hat{Z}_1, \hat{K}_1) \\[\smallskipamount]

  \quad \quad - \{a_0 (x - \hat{X}) + a_1(y - \hat{Y}_1) + a_2(\Lambda - \Lambda_1) + a_3(w - \hat{W}_1) + a_4(z - \hat{Z}_1) \\[\smallskipamount]

  \quad \quad + \int_{\mb{R}} a_5(\zeta) (k - \hat{K}_1) \nu(d \zeta))\}.
\end{array}
 \end{equation}

Then, by equation~\eqref{eq: superdiff}

\begin{equation}
\label{eq: phi}
\begin{array}{lll}
\phi_1(x,y, \Lambda, w,z,k) &\leq& 0 \mbox{ for all } x,y, \Lambda, w,z,k, \\[\smallskipamount]
\phi_1(\hat{X}, \hat{Y}_1, \hat{\Lambda}_1, \hat{W}_1, \hat{Z}_1, \hat{K}_1) &=& 0 \mbox{ (by definition)}.
\end{array}
\end{equation}

Therefore, by differentiating equation~\eqref{eq: phi_1} and using equation~\eqref{eq: phi}, we find that

\[
\begin{array}{llllll}
a_0 &=& \frac{\partial \hat{\mc{H}}_1}{\partial x}(\hat{X}, \hat{Y}_1, \hat{\Lambda}_1, \hat{W}_1, \hat{Z}_1, \hat{K}_1) &=& \frac{\partial \hat{H}_1}{\partial x} \\[\smallskipamount]

a_1 &=& \frac{\partial \hat{\mc{H}}_1}{\partial y_1}(\hat{X}, \hat{Y}_1, \hat{\Lambda}_1, \hat{W}_1, \hat{Z}_1, \hat{K}_1) &=& \frac{\partial \hat{H}_1}{\partial y_1} \\[\smallskipamount]

a_2 &=& \frac{\partial \hat{\mc{H}}_1}{\partial \Lambda_1}(\hat{X}, \hat{Y}_1, \hat{\Lambda}_1, \hat{W}_1, \hat{Z}_1, \hat{K}_1) &=& \frac{\partial \hat{H}_1}{\partial \Lambda_1} \\[\smallskipamount]

a_3 &=& \frac{\partial \hat{\mc{H}}_1}{\partial w_1}(\hat{X}, \hat{Y}_1, \hat{\Lambda}_1, \hat{W}_1, \hat{Z}_1, \hat{K}_1) &=& \frac{\partial \hat{H}_1}{\partial w_1}

\\[\smallskipamount]

a_4 &=& \frac{\partial \hat{\mc{H}}_1}{\partial z_1}(\hat{X}, \hat{Y}_1, \hat{\Lambda}_1, \hat{W}_1, \hat{Z}_1, \hat{K}_1) &=& \frac{\partial \hat{H}_1}{\partial z_1}

\\[\smallskipamount]

a_5 &=& \nabla_{k_1} \hat{\mc{H}}_1(\hat{X}, \hat{Y}_1, \hat{\Lambda}_1, \hat{W}_1, \hat{Z}_1, \hat{K}_1) &=& \nabla_{k_1} \hat{H}_1.
\end{array}
\]

Therefore, it follows from this, equation~\eqref{eq: delta_2} and equation~\eqref{eq: phi} that

\[
 \Delta = \phi(X(t), Y_1(t), \Lambda_1(t), W_1(t), Z_1(t), K_1(t, \cdot)) \leq 0
\]
\noindent where the final inequality follows since $\hat{\mc{H}}_1$ is concave.

This means that $J_1(u_1, \hat{u}_2) \leq J_1(\hat{u}_1, \hat{u}_2)$ for all $u_1 \in \mc{A}_1$.

In a similar way, one can prove that $J_2(\hat{u}_1, u_2) \leq J_2(\hat{u}_1, \hat{u}_2)$ for all $u_2 \in \mc{A}_2$. This completes the proof that $(\hat{u}_1, \hat{u}_2)$ is a Nash-equilibrium.


\end{proof}

\section{Necessary maximum principle for FBSDE games with delay and noisy memory}
\label{sec: necessary}

In the following, we need some additional assumptions and notation:

\begin{itemize}
 \item{For all $t_0 \in [0,T]$ and all bounded $\mc{E}_i(t)$-measurable random variables $\alpha_i(\omega)$, the control

 \begin{equation}
 \label{assumption: A1}
  \beta_i(t) := \boldsymbol{1}_{(t_0,T)}(t) \alpha_i(\omega) \mbox{ is in } \mc{A}_i \mbox{ for } i=1,2.
 \end{equation}}

 \item{For all $u_i, \beta_i \in \mc{A}_i$ with $\beta_i$ bounded, there exists $\kappa_i > 0$ such that the control

 \begin{equation}
 \label{assumption: A2}
  u_i(t) + s \beta_i(t) \mbox{ for } t \in [0,T]
 \end{equation}
 \noindent belongs to $\mc{A}_i$ for all $s \in (-\kappa_i, \kappa_i)$, $i=1,2$.
}

\item{Also, assume that the following derivative processes exist and belong to $L^2([0,T] \times \Omega)$:

\begin{equation}
\label{assumption: A3}
 \begin{array}{llll}
  x_1(t) &=& \frac{d }{d s}X^{(u_1 + s\beta_1,u_2)}(t)|_{s=0}, \\[\smallskipamount]
  y_1(t) &=& \frac{d }{d s}Y_1^{(u_1 + s\beta_1,u_2)}(t)|_{s=0}, \\[\smallskipamount]
  \tilde{\Lambda}_1(t) &=& \frac{d }{d s}\Lambda_1^{(u_1 + s\beta_1,u_2)}(t)|_{s=0}, \\[\smallskipamount]
  w_1(t) &=& \frac{d }{d s}W_1^{(u_1 + s\beta_1,u_2)}(t)|_{s=0}, \\[\smallskipamount]
  z_1(t) &=& \frac{d }{d s}Z_1^{(u_1 + s\beta_1,u_2)}(t)|_{s=0}, \\[\smallskipamount]
  k_1(t) &=& \frac{d }{d s}K_1^{(u_1 + s\beta_1,u_2)}(t)|_{s=0}, \\[\smallskipamount]
 \end{array}
\end{equation}
\noindent and similarly for $x_2(t) = \frac{d }{d s}X^{(u_1 ,u_2+ s\beta_2)}(t)|_{s=0}$ etc. Notice that $x_i(0)=0$ for $i=1,2$ since $X(0)=x$.}
\end{itemize}

\bigskip

If these assumptions hold, we can prove a necessary maximum principle for our noisy memory FBSDE game. The proof of the following theorem is based on the same idea as the proof of Theorem 2.2 in {\O}ksendal and Sulem~\cite{OS_artikkel}, however the presence of noisy memory in our problem requires some extra care.

\bigskip

\begin{theorem}
\label{thm: necessary}
Suppose that $u \in \mc{A}$ with corresponding solutions $X(t), Y_i(t),$ $\Lambda_i(t), W_i(t),$ $Z_i(t), K_i(t,\zeta),$ $\lambda_i(t), p_i(t),$ $q_i(t), r_i(t,\zeta)$, $i=1,2$, of equations~\eqref{eq: FSDE}, \eqref{eq: BSDE}, \eqref{eq: FSDE_adjoint} and \eqref{eq: BSDE_adjoint}. Also, assume that conditions \eqref{assumption: A1}-\eqref{assumption: A3} hold. Then, the following are equivalent:

\begin{enumerate}
 \item[$(i)$] $\frac{\partial}{\partial s} J_1(u_1 + s \beta_1, u_2)|_{s=0} = \frac{\partial}{\partial s} J_2(u_1, u_2 + s \beta_2)|_{s=0} =0$ for all bounded $\beta_1 \in \mc{A}_1, \beta_2 \in \mc{A}_2$.
 \item[$(ii)$] $E[\frac{\partial  H_1(t, X(t), \bm{Y}(t),\bm{\Lambda}(t), W_1(t), Z_1(t), K_1(t,\cdot), v_1, u_2(t), \lambda_1(t), p_1(t), q_1(t), r_1(t, \cdot))}{\partial v_1}]|_{v_1=u_1(t)} \\[\smallskipamount]$
 $= E[\frac{\partial  H_2(t, X(t), \bm{Y}(t),\bm{\Lambda}(t), W_2(t), Z_2(t), K_2(t,\cdot), u_1(t), v_2, \lambda_2(t), p_2(t), q_2(t), r_2(t, \cdot))}{\partial v_2}]|_{v_2=u_2(t)} \\[\smallskipamount]$
 $=0$.
\end{enumerate}

\end{theorem}

\begin{proof}
 We only prove that $\frac{\partial}{\partial s} J_1(u_1 + s \beta_1, u_2)|_{s=0} =0$ for all bounded $\beta_1 \in \mc{A}_1$ is equivalent to

$\\[\smallskipamount]
E[\frac{\partial  H_1(t, X(t), \bm{Y}(t),\bm{\Lambda}(t), W_1(t), Z_1(t), K_1(t,\cdot), v_1, u_2(t), \lambda_1(t), p_1(t), q_1(t), r_1(t, \cdot))}{\partial v_1}]|_{v_1=u_1(t)} =0. \\[\smallskipamount]$

The remaining part of the theorem (i.e., the same statement for $J_2$ and $H_2$) is proved in a similar way.

Note that, by the definition of $J_1$ and by interchanging differentiation and integration,

\[
\begin{array}{lll}
D_1 &:=& \frac{\partial}{\partial s} J_1(u_1 + s \beta_1, u_2)|_{s=0} \\[\smallskipamount]
&=& E[\int_0^T \{\frac{\partial f_1}{\partial x}(t) x_1(t) + \frac{\partial f_1}{\partial y}(t) y_1(t) + \frac{\partial f_1}{\partial \Lambda}(t) \tilde{\Lambda}_1(t) \frac{\partial f_1}{\partial u_1}(t) \beta_1(t) \}dt \\[\smallskipamount]
&&+ \varphi_1'(X(T))x_1(T) + \phi_1'(W_1(0))w_1(0)].
\end{array}
\]

We study the different parts of $D_1$ separately. First, by the It{\^o} product rule, the adjoint BSDE~\eqref{eq: BSDE_adjoint} and the definition of $x_1(t)$,

\begin{equation}
\label{eq: necessary_mellom_1}
 \begin{array}{rlll}
  I_1 :=& E[\varphi_1'(X(T))x_1(T)] \\[\smallskipamount]
  =& E[p_1(T)x_1(T)] - E[h_1'(X(T)) \lambda_1(T) x_1(T)] \\[\smallskipamount]
  =& E[p_1(0) x_1(0)] + E[\int_0^T p_1(t) dx_1(t) + \int_0^T x_1(t) dp_1(t) \\[\smallskipamount]
  &+ \int_0^T d[p_1,x_1](t)] - E[h_1'(X(T)) \lambda_1(T) x_1(T)] \\[\smallskipamount]
  =& E[ \int_0^T p_1(t) \big( \frac{\partial b}{\partial x}(t)x_1(t) + \frac{\partial b}{\partial y_1}(t) y_1(t) + \frac{\partial b}{\partial \Lambda_1}(t)\tilde{\Lambda}_1(t) + \frac{\partial b}{\partial u_1}(t) \beta_1(t)   \big)dt] \\[\smallskipamount]
  & + E[\int_0^T x_1(t) E[\mu_1(t) | \mc{F}_t]dt]\\[\smallskipamount]
  &+ E[\int_0^T q_1(t) \big( \frac{\partial \sigma}{\partial x}(t)x_1(t) + \frac{\partial \sigma}{\partial y_1}(t)y_1(t) + \frac{\partial \sigma}{\partial \Lambda_1}\tilde{\Lambda}_1(t) + \frac{\partial \sigma}{\partial u_1}(t) \beta_1(t)  \big) dt] \\[\smallskipamount]
  &+ E[\int_0^T \int_{\mb{R}} r_1(t, \zeta) \big( \frac{\partial \gamma}{\partial x}(t)x_1(t) + \frac{\partial \gamma}{\partial y_1}(t)y_1(t) + \frac{\partial \gamma}{\partial \Lambda_1}\tilde{\Lambda}_1(t) + \frac{\partial \gamma}{\partial u_1}(t) \beta_1(t)  \big) d\nu(\zeta) dt] \\[\smallskipamount]
  &- E[h_1'(X(T)) \lambda_1(T) x_1(T)].
 \end{array}
\end{equation}

Also, by the FSDE~\eqref{eq: FSDE_adjoint}, the BSDE~\eqref{eq: BSDE}, the definition of $x_1(t)$ and the It{\^o} product rule,

\begin{equation}
\label{eq: necessary_mellom_2}
 \begin{array}{rllll}
  I_2 :=& E[\phi_1'(W_1(0)) w_1(0)] \\[\smallskipamount]
  =& E[\lambda_1(0)w_1(0)] \\[\smallskipamount]
  =& E[\lambda_1(T)w_1(T)] - E[\int_0^T \lambda_1(t) dw_1(t) + \int_0^T w_1(t) d\lambda_1(t) \\[\smallskipamount]
  &+ \int_0^T z_1(t) \frac{\partial H_1}{\partial z_1}(t)dt + \int_0^T \int_{\mb{R}} \nabla_{k_1} H_1(t, \zeta) k_1(t,\zeta) \nu(d \zeta) dt] \\[\smallskipamount]
  =& E[\lambda_1(T)h_1'(X(T))x_1(T)] + E[\int_0^T \lambda_1(t) \big( \frac{\partial g_1}{\partial x}(t)x_1(t) + \frac{\partial g_1}{\partial y_1}(t)y_1(t) \\[\smallskipamount]
  &+ \frac{\partial g_1}{\partial \Lambda_1}(t)\tilde{\Lambda}(t) + \frac{\partial g_1}{\partial w_1}(t)w_1(t) + \frac{\partial g_1}{\partial z_1}(t)z_1(t) + \nabla_{k_1} g_1(t)k_1(t) \\[\smallskipamount]
  & + \frac{\partial g_1}{\partial u_1}(t) \beta_1(t) \big)dt] - E[\int_0^T \frac{\partial H_1}{\partial w_1}(t)w_1(t) dt] \\[\smallskipamount]
  & -E[\int_0^T z_1(t) \frac{\partial H_1}{\partial z_1}(t)dt + \int_0^T \int_{\mb{R}} \nabla_k H_1(t, \zeta) k_1(t,\zeta) \nu(d \zeta) dt].
 \end{array}
\end{equation}

By the definition of $D_1$ as well as equations~\eqref{eq: necessary_mellom_1} and \eqref{eq: necessary_mellom_2},

\begin{equation}
\label{eq: D_1}
\begin{array}{llll}
 D_1 &=& A + E[\int_0^T \beta_1(t) \big( \frac{\partial f_1}{\partial u_1}(t) + \frac{\partial b}{\partial u_1}(t)p_1(t) + \frac{\partial \sigma}{\partial u_1}(t) q_1(t) + \frac{\partial \gamma}{\partial u_1}(t)r_1(t) \\[\smallskipamount]
 &&+ \frac{\partial g_1}{\partial u_1}(t)\lambda_1(t) \big)dt] + E[\int_0^T w_1(t) \{ - \frac{\partial H_1}{\partial w_1}(t) + \frac{\partial g_1}{\partial w_1}(t) \lambda_1(t)\}dt \\[\smallskipamount]
 &&+ \int_0^T z_1(t) \{- \frac{\partial H_1}{\partial z_1}(t) + \frac{\partial g_1}{\partial x}(t) \lambda_1(t) \}dt \\[\smallskipamount]
&&+ \int_0^T k_1(t) \{-\nabla_{k_1} H_1(t) + \nabla_k g_1(t) \lambda_1(t) \}dt]
\end{array}
\end{equation}

\noindent where

\begin{equation}
\label{eq: A}
\begin{array}{lll}
 A &:=& E[\int_0^T x_1(t) \{\frac{\partial f_1}{\partial x}(t) + \frac{\partial b}{\partial x}(t)p_1(t) + E[\mu_1(t) | \mc{F}_t] + \frac{\partial \sigma}{\partial x}(t) q_1(t) \\[\smallskipamount]
 &&+\frac{\partial \gamma}{\partial x}(t) r_1(t) + \frac{\partial g_1}{\partial x}(t) \lambda_1(t)  \}dt + \int_0^T y_1(t) \{\frac{\partial f_1}{\partial y_1}(t) + \frac{\partial b}{\partial y_1}(t)p_1(t)  \\[\smallskipamount]
 &&+ \frac{\partial \sigma}{\partial y_1}(t) q_1(t) +\frac{\partial \gamma}{\partial y_1}(t) r_1(t) + \frac{\partial g_1}{\partial y_1}(t) \lambda_1(t)  \}dt + \int_0^T \tilde{\Lambda}_1(t) \{\frac{\partial f_1}{\partial \Lambda_1}(t) \\[\smallskipamount]
 &&+ \frac{\partial b}{\partial \Lambda_1}(t)p_1(t) +  \frac{\partial \sigma}{\partial \Lambda_1}(t) q_1(t) +\frac{\partial \gamma}{\partial \Lambda_1}(t) r_1(t) + \frac{\partial g_1}{\partial \Lambda_1}(t) \lambda_1(t)  \}dt] \\[\smallskipamount]
 &=& E[\int_0^T x_1(t) \{\frac{\partial H_1}{\partial x}(t) + E[\mu_1(t) | \mc{F}_t] \} dt] + E[\int_0^T y_1(t) \frac{\partial H_1}{\partial y_1}(t)] \\[\smallskipamount]
 &&+ E[\int_0^T \tilde{\Lambda}_1(t) \frac{\partial H_1}{\partial \Lambda_1}(t)].
 \end{array}
\end{equation}

Then, by using the definition of the Hamiltonian $H_1$, see equation~\eqref{eq: Hamiltonians}, we see that everything inside the curly brackets in equation~\eqref{eq: D_1} is equal to zero. Hence,

\[
 D_1 = A + E[\int_0^T \beta_1(t) \frac{\partial H_1}{\partial u_1}(t)dt].
\]

Recall that from the definitions of $y_1$ and $\tilde{\Lambda}_1$,

\[
 y_1(t) = x_1(t-\delta_1) \mbox{ and } \tilde{\Lambda}_1(t) = \int_{t-\delta_1}^t x_1(u) dB(u).
\]

This implies, by change of variables

\[
 \begin{array}{lll}
  E[\int_0^T y_1(t) \frac{\partial H_1}{\partial y_1}(t)] &=& E[\int_0^T x_1(t - \delta_1) \frac{\partial H_1}{\partial y_1}(t) dt] \\[\smallskipamount]
  &=& \int_{-\delta_1}^{T-\delta_1} x_1(u) \frac{\partial H_1}{\partial y_1}(u + \delta_1)du] \\[\smallskipamount]
  &=& E[\int_0^T x_1(u) \boldsymbol{1}_{[0,T-\delta_1]}(u) \frac{\partial H_1}{\partial y_1}(u + \delta_1) du].
 \end{array}
\]

Also, by the duality formula for Malliavin derivatives (see Di Nunno et al.~\cite{DiNunno}) and changing the order of integration

\[
 \begin{array}{lll}
  E[\int_0^T \tilde{\Lambda}_1(t) \frac{\partial H_1}{\partial \Lambda_1}(t)] &=& E[\int_0^T \int_{t-\delta_1}^{t} x_1(u) dB(u) \frac{\partial H_1}{\partial \Lambda_1}(t) dt] \\[\smallskipamount]
  &=& E[\int_0^T \int_{t-\delta_1}^t E[D_u(\frac{\partial H_1}{\partial \Lambda_1}(t)) | \mc{F}_u] x_1(u) du \mbox{ } dt] \\[\smallskipamount]
  &=& E[\int_0^T \int_u^{u + \delta_1} E[D_u(\frac{\partial H_1}{\partial \Lambda_1}(t)) | \mc{F}_u] \boldsymbol{1}_{[0,T]}(t) dt \mbox{ } x_1(u) \mbox{ } du ].
 \end{array}
\]

But, from the definition of $\mu_1$,

\[
 \begin{array}{lll}
  E[\int_0^T x_1(t) E[\mu_1(t) | \mc{F}_t] dt ] &=& E[\int_0^T E[x_1(t) \mu_1(t) | \mc{F}_t] \mbox{ } dt] \\[\smallskipamount]
  &=& E[ \int_0^T E[x_1(t) \{ -\frac{\partial H_1}{\partial x}(t) - \frac{\partial H_1}{\partial y_1}(t + \delta_1)\boldsymbol{1}_{[0,T-\delta_1]} \\[\smallskipamount]
  &&- \int_t^{t+\delta_1} D_t[\frac{\partial H_1}{\partial \Lambda_1}(s)] \boldsymbol{1}_{[0,T]}(s)ds \} | \mc{F}_t] dt].
 \end{array}
\]

So, by the rule of double expectation and the calculations above, $A=0$. This implies that $D_1= E[\int_0^T \beta_1(t) \frac{\partial H_1}{\partial u_1}(t)dt]$, so

\[
  \frac{\partial}{\partial s} J_1(u_1 + s \beta_1, u_2)|_{s=0} = E[\int_0^T \beta_1(t) \frac{\partial H_1}{\partial u_1}(t)dt]
\]

\noindent which was what we wanted to prove.

\end{proof}

\section{Solution of the noisy memory FBSDE}
\label{sec: FBSDE}

In this section, we consider a slightly simplified version of the system of noisy memory FBSDEs in equations~\eqref{eq: FSDE_adjoint} and \eqref{eq: BSDE_adjoint}. Instead, consider the following \emph{noisy memory FBSDE}:

FSDE in $\lambda$,
\begin{equation}
 \label{eq: noisyFBSDE}
 \begin{array}{llll}
 d\lambda(t) &=& \frac{\partial H}{\partial w}(t)dt + \frac{\partial H}{\partial z}(t)dB(t) + \int_{\mb{R}} \nabla_k H(t, \zeta) \tilde{N}(dt,d \zeta) \\[\smallskipamount]
 \lambda(0) &=& \phi'(W(0)).
 \end{array}
\end{equation}

BSDE in $p, q$ and $r$,
\begin{equation}
 \label{eq: noisyBSDE}
 \begin{array}{llll}
  dp(t) &=& -E[\mu(t) | \mc{F}_t]dt + q(t)dB(t) + \int_{\mb{R}} r(t,\zeta)\tilde{N}(dt, d\zeta) \\[\smallskipamount]
  p(T) &=& \varphi'(X(T)) + h'(X(T))\lambda(T)
 \end{array}
\end{equation}

\noindent where

\[
\begin{array}{lll}
 H(t,x, y_1, y_2, \Lambda_1, \Lambda_2, w, z, k, u_1, u_2, \lambda, p, q, r) \\[\smallskipamount]
\quad \quad \quad = f(t,x, y, \Lambda, u_1, u_2) + \lambda g(t,x, y_1, y_2, \Lambda_1, \Lambda_2, w, z, k, u_1, u_2) \\[\smallskipamount]
\quad \quad \quad + p b(t,x, y_1, y_2, \Lambda_1, \Lambda_2, u_1, u_2) + q \sigma(t,x, y_1, y_2, \Lambda_1, \Lambda_2, u_1, u_2) \\[\smallskipamount]
\quad \quad \quad + \int_{\mb{R}} r(\zeta) \gamma (t,x, y_1, y_2, \Lambda_1, \Lambda_2, u_1, u_2, \zeta) \nu(d \zeta)
\end{array}
\]

and

\[
\mu(t) = \frac{\partial H}{\partial x}(t) + \frac{\partial H}{\partial y}(t + \delta) \boldsymbol{1}_{[0,T-\delta]}(t) + \int_t^{t+\delta} E[D_t[\frac{\partial H}{\partial \Lambda}(s)] | \mc{F}_t] \boldsymbol{1}_{[0,T]}(s)ds.
\]

Note that the set of equations~\eqref{eq: FSDE_adjoint} and \eqref{eq: BSDE_adjoint} are two such systems such as \eqref{eq: noisyFBSDE}-\eqref{eq: noisyBSDE} involving the same $X$ process as well as the same controls $u_1, u_2$.

Also, consider the following system consisting of an FSDE and two BSDEs:

FSDE in $\lambda$,
\begin{equation}
 \label{eq: noisyFBSDE2}
 \begin{array}{llll}
 d\tilde{\lambda}(t) &=& \frac{\partial \mc{H}}{\partial w}(t)dt + \frac{\partial \mc{H}}{\partial z}(t)dB(t) + \int_{\mb{R}} \nabla_k \mc{H}(t, \zeta) \tilde{N}(dt,d \zeta) \\[\smallskipamount]
 \tilde{\lambda}(0) &=& \phi'(W(0)).
 \end{array}
\end{equation}

BSDE in $p_1, q_1$ and $r_1$,
\begin{equation}
 \label{eq: noisyBSDE1}
 \begin{array}{llll}
  dp_1(t) &=& -E[\mu_1(t) | \mc{F}_t]dt + q_1(t)dB(t) + \int_{\mb{R}} r_1(t,\zeta)\tilde{N}(dt, d\zeta) \\[\smallskipamount]
  p_1(T) &=& \varphi'(X(T)) + h'(X(T))\tilde{\lambda}(T).
 \end{array}
\end{equation}

BSDE in $p_2, q_2$ and $r_2$,
\begin{equation}
 \label{eq: noisyBSDE2}
 \begin{array}{llll}
  dp_2(t) &=& -E[\mu_2(t) | \mc{F}_t]dt + q_2(t)dB(t) + \int_{\mb{R}} r_2(t,\zeta)\tilde{N}(dt, d\zeta) \\[\smallskipamount]
  p_2(T) &=& 0
 \end{array}
\end{equation}

\noindent where

\begin{equation}
\label{eq: cal_H}   
\begin{array}{lll}
 \mc{H}(t,x, y_1, y_2, \Lambda_1, \Lambda_2, w, z, k, u_1, u_2, \tilde{\lambda}, p_1, p_2, q_1, q_2, r_1, r_2) \\
\hspace{2cm} = q_2(t)x + H(t,x, y_1, y_2, \Lambda_1, \Lambda_2, w, z, k, u_1, u_2, \tilde{\lambda}, p_1, q_1, r_1),
\end{array}
\end{equation}

\[
\mu_1(t) = q_2(t) +\frac{\partial H}{\partial x}(t) + \frac{\partial H}{\partial y}(t + \delta) \boldsymbol{1}_{[0,T-\delta]}(t)
\]

\noindent and

\[
\mu_2(t) = \frac{\partial H}{\partial \Lambda}(t) - \frac{\partial H}{\partial \Lambda}(t + \delta) \boldsymbol{1}_{[0,T-\delta]}(t).
\]

\noindent Note that $\frac{\partial \mc{H}}{\partial \Lambda}(t) = \frac{\partial H}{\partial \Lambda}(t)$, $\frac{\partial \mc{H}}{\partial \Lambda}(t) = q_2(t) + \frac{\partial H}{\partial \Lambda}(t)$ and $\frac{\partial \mc{H}}{\partial y}(t) = \frac{\partial H}{\partial y}(t)$. Hence, equations~\eqref{eq: noisyFBSDE} and \eqref{eq: noisyFBSDE2} are structurally equal.

Then, by similar techniques as in Dahl et al.~\cite{Dahl}, we can show the following theorem:

\begin{theorem}
 \label{thm: noisyFBSDE_solution}
 Assume that $(p_i, q_i, r_i)$ for $i=1,2$ and $\tilde{\lambda}$ solve the FBSDE system \eqref{eq: noisyFBSDE2}-\eqref{eq: noisyBSDE2}. Define $\lambda = \tilde{\lambda}$, $p(t) = p_1(t)$, $q(t)=q_1(t)$ and $r(t, \cdot) =r_1(t, \cdot)$ and assume that $E[\int_0^T (\frac{\partial H(t)}{\partial z})^2]dt < \infty$. Then, $(p,q,r,\lambda)$ solves the noisy memory FBSDE \eqref{eq: noisyFBSDE}-\eqref{eq: noisyBSDE} and

 \[
  q_2(t) = \int_t^{t+\delta} E[D_t[\frac{\partial H}{\partial \Lambda}(s)] | \mc{F}_t]ds.
 \]

\end{theorem}

\begin{proof}
 The jump terms do not make a difference here, so assume for simplicity that $r=r_1=r_2=0$ everywhere.

 In general, we know that if $dp_2(t) = -\theta(t,p_2, q_2)dt + q_2(t)dB(t)$, $p_2(T)=F$, then

 \begin{equation}
 \label{eq: Malliavin_BSDE}
q_2(t) = D_t p_2(t).
 \end{equation}

 \noindent Now, note that the solution $p_2$ of the BSDE~\eqref{eq: noisyBSDE2} can be written

 \[
 \begin{array}{lll}
  p_2(t) &=& -E[\int_t^T E[\mu_2(s) | \mc{F}_s]ds | \mc{F}_t] \\[\smallskipamount]
  &=& -\int_t^T E[\mu_2(s) | \mc{F}_t] ds \\[\smallskipamount]
  &=& -\int_t^T E[\frac{\partial H}{\partial \Lambda}(t) - \frac{\partial H}{\partial \Lambda}(t + \delta) \boldsymbol{1}_{[0,T-\delta]}(t) | \mc{F}_t]ds \\[\smallskipamount]
  &=& -\int_t^{t+\delta} E[\frac{\partial H(s)}{\partial \Lambda} | \mc{F}_t] \boldsymbol{1}_{[0,T]}(s) ds
  \end{array}
 \]

 where the equalities follow from Fubini's theorem, the rule of double expectation, the definition of $\mu_2$ and a change of variables. Hence, by equation \eqref{eq: Malliavin_BSDE}:

 \[
 \begin{array}{llll}
  q_2(t) &=& D_t p_2(t) \\[\smallskipamount]
  &=& D_t[\int_t^{t+\delta} E[\frac{\partial H(s)}{\partial \Lambda} | \mc{F}_t] \boldsymbol{1}_{[0,T]}(s)] ds  \\[\smallskipamount]
  &=& \int_t^{t+\delta}  E[ D_t(\frac{\partial H(s)}{\partial \Lambda}) | \mc{F}_t] \boldsymbol{1}_{[0,T]}(s)  ds
  \end{array}
 \]
 \noindent which is part of what we wanted to prove.

 By inserting this expression for $q_2$ into the definition of $\mu_1$, we see that

 \[
  \mu_1(t) = \int_t^{t+\delta} E[D_t[\frac{\partial H(s)}{\partial \Lambda}]| \mc{F}_t] \boldsymbol{1}_{[0,T]}(s) ds + \frac{\partial H(t)}{\partial x} + \frac{\partial H(t + \delta)}{\partial y} \boldsymbol{1}_{[0,T]}(t+\delta).
 \]

 Hence, we see that the BSDE~\eqref{eq: noisyBSDE1} is the same as \eqref{eq: noisyBSDE}, so they have the same solution. This completes the proof of the theorem.

\end{proof}

We can also prove the following converse result.

\begin{theorem}
\label{eq: NoisyFBSDE_solution_converse}
If $p,q,r, \lambda$ solve the FBSDE~\eqref{eq: noisyFBSDE}-\eqref{eq: noisyBSDE} and we define $\tilde{\lambda} = \lambda$, $p_1=p$, $q_1=q$, $r_1=r$ and
\[
\begin{array}{lll}
 p_2(t) &=& \int_t^{t+\delta} E[\frac{\partial H}{\partial \Lambda}(s) | \mc{F}_t] \boldsymbol{1}_{[0,T-\delta]}(s)ds \\[\smallskipamount]
 q_2(t) &=& \int_t^{t+\delta} E[D_t[\frac{\partial H}{\partial \Lambda}(s)] | \mc{F}_t] \boldsymbol{1}_{[0,T-\delta]}(s)ds \\[\smallskipamount]
 r_2(t, \cdot) &=& 0.
\end{array}
\]

Then, $(p_i, q_i, r_i)$ for $i=1,2$ and $\tilde{\lambda}$ solve the system of equations \eqref{eq: noisyFBSDE2}-\eqref{eq: noisyBSDE2}.
\end{theorem}

\begin{proof}
Again, the jump parts make no crucial difference, so we consider the no-jump situation for simplicity.

It is clear that equation~\eqref{eq: noisyFBSDE2} holds from the assumptions above (from the definition of $\mc{H}$, see \eqref{eq: cal_H}). Also, the BSDE~\eqref{eq: noisyBSDE1} holds: Clearly, the terminal condition holds, and by the computations in the proof of Theorem~\ref{thm: noisyFBSDE_solution}, the remaining part of equation~\eqref{eq: noisyBSDE1} also holds. Therefore, it only remains to prove that the BSDE~\eqref{eq: noisyBSDE2} holds.

By the It{\^o} isometry and the Clark-Ocone formula,

\[
 \begin{array}{lll}
 E[\int_0^T E[D_s(\frac{\partial H(r)}{\partial \Lambda}) | \mc{F}_s]^2 ds] &=& E[(\int_0^T E[D_s \frac{\partial H(r)}{\partial \Lambda} | \mc{F}_s] dB_s)^2] \\[\smallskipamount]
 &=& E[(\frac{\partial H}{\partial \Lambda}(r))^2 - E[\frac{\partial H}{\partial \Lambda}(r)]^2].
 \end{array}
\]

\noindent Hence,
\[
 \begin{array}{llll}
  \int_0^T E[\int_0^T E[D_s(\frac{\partial H(r)}{\partial \Lambda}) | \mc{F}_s]^2 ds] ^{\frac{1}{2}}dr = \int_0^T (E[\frac{\partial H}{\partial \Lambda}(r)^2] - E[\frac{\partial H}{\partial \Lambda}(r)]^2)^{\frac{1}{2}} dt < \infty.
 \end{array}
\]

\noindent Note that from the Clark-Ocone theorem,

\[
 \frac{\partial H(r)}{\partial \Lambda} = E[\frac{\partial H(r)}{\partial \Lambda} | \mc{F}_t] + \int_t^r E[ D_s(\frac{\partial H(r)}{\partial \Lambda}) | \mc{F}_s] dB(s).
\]

\noindent Therefore, by the definition of $q_2$ in the theorem and the Fubini theorem

\[
 \begin{array}{llll}
  \int_t^T q_2(s) dB(s) &=& \int_t^T \int_t^T E[D_s (\frac{\partial H(r)}{\partial \Lambda}) | \mc{F}_s] \boldsymbol{1}_{[s, s+\delta]}(r) dr dB(s) \\[\smallskipamount]

  &=& \int_t^T \int_t^T E[D_s(\frac{\partial H(r)}{\partial \Lambda}) | \mc{F}_s] \boldsymbol{1}_{[r-\delta,r]}(s) dB(s) dr.
\end{array}
\]

\noindent By some algebra and the Clark-Ocone theorem~\eqref{Clark-Ocone}, 

\[
\begin{array}{lll}
\int_t^T \int_t^T E[D_s(\frac{\partial H(r)}{\partial \Lambda}) | \mc{F}_s] \boldsymbol{1}_{[r-\delta,r]}(s) dB(s) dr  &=& \int_t^T \int_{r-\delta}^r E[D_s (\frac{\partial H(r)}{\partial \Lambda}) | \mc{F}_s] dB(s) dr \\[\smallskipamount]

  &=& \int_t^T (\frac{\partial H(r)}{\partial \Lambda} - E[\frac{\partial H(r)}{\partial \Lambda} | \mc{F}_{r-\delta}]) dr 
  
\end{array}
\]

\noindent By splitting the integrals and using change of variables (twice) as well as some algebra,

\[
\begin{array}{llll}
  &=& \int_t^T \frac{\partial H(s)}{\partial \Lambda} ds - \int_{t-\delta}^{T-\delta} E[\frac{\partial H(s + \delta)}{\partial \Lambda} | \mc{F}_s] ds \\[\smallskipamount]

  &=& \int_t^T \frac{\partial H(s)}{\partial \Lambda} ds - \int_{t}^{T} E[\frac{\partial H(s + \delta)}{\partial \Lambda} | \mc{F}_s] \boldsymbol{1}_{[0,T-\delta]}(s) ds \\[\smallskipamount]
  && - \int_t^{t+\delta} E[\frac{\partial H(s)}{\partial \Lambda} | \mc{F}_t] \boldsymbol{1}_{[0,T-\delta]}(s) ds \\[\smallskipamount]

  &=& \int_t^T E[\frac{\partial H(s)}{\partial \Lambda}  - \frac{\partial H(s+\delta)}{\partial \Lambda} \boldsymbol{1}_{[0,T-\delta]}(s) |\mc{F}_s]ds - p_2(t).

 \end{array}
\]

\noindent This proves that the BSDE~\eqref{eq: noisyBSDE2} holds as well.

\end{proof}

Now, we have expressed the solution of the Malliavin FBSDE via the solution of the ``double'' FBSDE system~\eqref{eq: noisyFBSDE2}-\eqref{eq: noisyBSDE2}. What kind of system of equations is this? The system consists of two connected BSDEs in $(p_1,q_1,r_1)$ and $(p_2, q_2, r_2)$ respectively, and these are again connected to a FBSDE in $\lambda$. However, from equation \eqref{eq: noisyBSDE2} and the definition of $\mu_2$, we see that the right hand side of \eqref{eq: noisyBSDE2} does not depend on $p_2$. Hence, the BSDE \eqref{eq: noisyBSDE2} can be rewritten

\[
 \begin{array}{lll}
  dp_2(t) &=& h(t, \lambda, p_1, q_1, r_1(\cdot))dt + q_2(t)dB(t) + \int_{\mb{R}} r_2(t, \zeta) \tilde{N}(dt, d\zeta) \\[\smallskipamount]
  p_2(T) &=& 0.
 \end{array}
\]

This can be solved to express $p_2$ using $\lambda, p_1, q_1$ and $r_1(\cdot)$ by letting $q_2(t) = r_2(t,\cdot) = 0$ for all $t$ and

\[
 p_2(t) = E[\int_t^T h(t, \lambda, p_1, q_1, r_1(\cdot)) dt | \mc{F}_t].
\]

Now, we can substitute this solution for $p_2(t)$ into the FBSDE system \eqref{eq: noisyFBSDE2}-\eqref{eq: noisyBSDE1}. The resulting set of equations is a regular system of time advanced FBSDEs with jumps. There are to the best of our knowledge, no general results on existence and uniqueness of such systems of FBSDEs. However, if we simplify by removing the jumps and there was no time-advanced part (i.e., no delay process $Y_i$ in the original FSDE~\eqref{eq: FSDE}), there are some results by Ma et al.~\cite{Ma}.

\section{Optimal consumption rate with respect to recursive utility}
\label{sec: application}

In this section, we apply the previous results to the problem of determining an optimal consumption rate with respect to recursive utility (see also {\O}ksendal and Sulem~\cite{OSartikkel2} and Dahl and {\O}ksendal~\cite{DahlOksendal}). Let $X(t) = X^c(t)$, where the consumption rate $c(t)$ is our control, and assume that

\begin{equation}
\label{eq: SDE_example}
 \begin{array}{lll}
  dX(t) &=& X(t)[\mu(t) dt + \sigma(t) dB(t) + \int_{\mb{R}} \gamma(t,\zeta) \tilde{N}(dt, d\zeta)] \\[\smallskipamount]
  &&- [c_1(t) + c_2(t)]X(t)dt, \\[\smallskipamount]
  X(0) &=& x > 0
 \end{array}
\end{equation}

\noindent and $W_i(t)$ is given by

\[
 \begin{array}{lll}
  dW_i(t) &=& -[\alpha_i(t)W_i(t)  + \eta_i(t) \ln(Y_i(t)) + \kappa_i(t) \ln(\Lambda_i(t))+ \ln(c_i(t)X(t))] \\[\smallskipamount]
  &&+ Z_i(t)dB(t) + \int_{\mb{R}} K_i(t, \zeta) \tilde{N}(dt, d\zeta) \\[\smallskipamount]
  W_i(T) &=& 0.
 \end{array}
\]

Let the performance functional be defined by $J_i(c_1, c_2) := W_{i}(0)$, i.e., $J_i$ is the recursive utility for player $i$. Also, assume that both players have full information, so $(\mc{E}_t^{(i)})_t = (\mc{F}_t)_t$ for $i=1,2$.

We would like to find a Nash equilibrium for this FBSDE game with delay. To do so we will use the maximum principle Theorem~\ref{thm: Suff-max-princ-FBSDE}. Note that $f_i = \varphi_i=h_i = 0$  and that $\psi_i(w)=w$ for $i=1,2$. The Hamiltonians are:

\[
\begin{array}{lll}
H_i(t, x, y_1, y_2, \Lambda_1, \Lambda_2, w_i, z_i, k_i, c_1, c_2, \lambda_i, p_i, q_i, r_i(\zeta)) \\[\smallskipamount]
\hspace{0.5cm} = \lambda_i (\alpha_i(t)w_i + \eta_i(t)\ln(y_i) + \ln(c_i x))  \\[\smallskipamount]
\hspace{0.5cm} + p_i(x \mu(t) - (c_1 + c_2)x)+ q_i \sigma(t) x + \int_{\mb{R}} x r_i(\zeta) \gamma(t,\zeta) \nu(d\zeta) \mbox{ for } i=1,2.
\end{array}
\]

The adjoint BSDEs are

\[
\begin{array}{lll}
dp_i(t) &=& E[\mu_i(t) | \mc{F}_t]dt + q_i(t) dB(t) + \int_{\mb{R}} r_i(t,\zeta) \tilde{N}(dt,d\zeta), \\[\smallskipamount]
p_i(T) &=& 0
\end{array}
\]
\noindent where
\[
\begin{array}{lll}
\mu_i(t) &=& -\frac{\lambda_i(t)}{X(t)} - \frac{\lambda_i(t + \delta_i) \eta_i(t+\delta_i)}{Y_i(t+\delta_i)} \boldsymbol{1}_{[0,T-\delta_i]}(t) - p_i(t)(\mu(t) - (c_1(t) + c_2(t)))  \\[\smallskipamount]
&&+ q_i(t)\sigma(t)+ \int_{\mb{R}} r_i(t,\zeta) \gamma(t,\zeta) \nu(d \zeta)
\end{array}
\]
\noindent for $i=1,2$. Note that by the definition of $Y_i$, $Y_i(t+\delta_i) = X(\{t+\delta_i\} - \delta_i) = X(t)$.

The adjoint BSDEs are linear, and the solutions are given by (see {\O}ksendal and Sulem~\cite{OksendalSulem_riskmin})

\begin{equation}
\label{eq: gamma_p}
\begin{array}{lll}
\Gamma_i(t)p_i(t) &=& E[\int_t^T ( \frac{\lambda_i(s)}{X(s)} +  \frac{\lambda_i(s + \delta_i) \eta_i(s+\delta_i)}{Y_i(s+\delta_i)}\boldsymbol{1}_{[0,T-\delta_i]}(s))\Gamma_i(s) ds | \mc{F}_t] \\[\smallskipamount]
&=& E[\int_t^T ( \frac{\lambda_i(s)}{X(s)} +  \frac{\lambda_i(s + \delta_i) \eta_i(s+\delta_i)}{X(s)}\boldsymbol{1}_{[0,T-\delta_i]}(s) ) \Gamma_i(s) ds | \mc{F}_t]
\end{array}
\end{equation}

\noindent where

\[
\begin{array}{lll}
 d \Gamma_i(t) &=& \Gamma_i(t)[(\mu(t) - (c_1(t) + c_2(t)))dt + \sigma(t)dB(t) + \int_{\mb{R}} \gamma(t, \zeta) \tilde{N}(dt, d\zeta)] \\[\smallskipamount]
 \Gamma_i(0) &=& 1 \mbox{ for } i=1,2.
\end{array}
\]

Note that by the SDE~\eqref{eq: SDE_example},

\begin{equation}
\label{eq: sammenheng_gamma_X}
x \Gamma_i(t) = X(t).
\end{equation}

Hence, by combining equations~\eqref{eq: gamma_p} and \eqref{eq: sammenheng_gamma_X}, we see that

\begin{equation}
 \label{eq: X_p}
\begin{array}{lll}
X(t)p_i(t) &=& E[\int_t^T ( \lambda_i(s) +  \lambda_i(s + \delta_i) \eta_i(s+\delta_i)\boldsymbol{1}_{[0,T-\delta_i]}(s) ) ds | \mc{F}_t].
\end{array}
\end{equation}

The adjoint FSDEs are

\[
\begin{array}{lll}
d\lambda_i(t) &=& \lambda_i(t) \alpha_i(t) dt \\[\smallskipamount]
\lambda_i(0) &=& 1, \mbox{ for } i=1,2.
\end{array}
\]

These are (non-stochastic) differential equation with solution $\lambda_i(t) = \exp(\int_0^t \alpha_i(s) ds)$ for $i=1,2$.

We maximize $H_i$ with respect to $c_i$. For $i=1,2$, the first order condition is:

\[
\hat{c}_i(t) = \frac{\lambda_i(t)}{p_i(t) X(t)}.
\]

By substituting equation~\eqref{eq: X_p} into this, we find (by the sufficient maximum principle, Theorem~\ref{thm: Suff-max-princ-FBSDE}) that the consumption rates leading to a Nash equilibrium for the recursive utility problem are given by:

\[
 c^*_i(t) = \frac{\lambda_i(t)}{E[\int_t^T ( \lambda_i(s) +  \lambda_i(s + \delta_i) \eta_i(s+\delta_i) \boldsymbol{1}_{[0,T-\delta_i]}(t) ) ds | \mc{F}_t]}.
\]
\noindent where $\lambda_i(t) = \exp(\int_0^t \alpha_i(s) ds)$ for $i=1,2$.

\section{Conclusion}

In this paper, we have analyzed a two-player stochastic game connected to a set of FBSDEs involving delay and noisy memory of the market process. We have derived sufficient and necessary maximum principles for a set of controls for the two players to be a Nash equilibrium in this game. We have also studied the associated FBSDE involving Malliavin derivatives, and connected this to a system of FBSDEs not involving Malliavin derivatives. Finally, we were able to derive a closed form Nash equilibrium solution to a game where the aim is to find the optimal consumption with respect to recursive utility.

\end{document}